\pdfoutput=1
\documentclass[a4paper]{amsart}

\usepackage{amssymb,amsmath,amsfonts,amsthm,epsfig,color}
\usepackage[all]{xypic}

\usepackage{float}
 \usepackage[hyphens]{url}

\newcommand{\R}{\mathbb{R}}
\newcommand{\Z}{\mathbb{Z}}
\newcommand{\N}{\mathbb{N}}

\renewcommand{\H}{\mathbb{H}}
\newcommand{\T}{\mathcal{T}}

\newcommand{\virg}[1]{``#1''}

\newcommand{\Span}[1]{\langle\, #1\, \rangle}
\renewcommand{\S}{\mathbb{S}}
\renewcommand{\P}{\mathcal{P}}
\newcommand{\Q}{\mathcal{Q}}

\newcommand{\Th}{^\textrm{th}}

\newtheorem{definition}{Definition}

\newtheorem{lemma}{Lemma}
\newtheorem{example}{Example}

\newtheorem{corol}{Corollary}

 \usepackage[colorlinks]{hyperref}
\begin{document}
\title[Vertex-to-edge duality]{On the vertex-to-edge duality between the Cayley  graph and the coset geometry of  von Dyck groups}  
\author[G. Moreno \and M.E. Stypa]%
{Giovanni Moreno* \and Monika Ewa Stypa**}

\newcommand{\acr}{\newline\indent}

\address{\llap{*\,}Mathematical Institute in Opava\acr
                   Silesian University in Opava\acr
                   Na Rybnicku 626/1\acr
                   746 01 Opava\acr
                   CZECH REPUBLIC}

\email{Giovanni.Moreno@math.slu.cz}

\address{\llap{**\,}Department of Mathematics\acr
                    Salerno University\acr
                    Via Ponte Don Melillo\acr
                    84084 Salerno\acr
                    ITALY}
\email{mstypa@unisa.it}

\thanks{The authors would like to thank P. Longobardi and C. Sica for their helpful suggestions, and also   the   referees for carefully reading the manuscript. The first author was supported by the project P201/12/G028 of the    Czech Republic Grant Agency  (GA \v CR). The second author was supported by the doctoral school  of the  University of Salerno.  }



\begin{abstract}
We prove that the Cayley  graph and the coset geometry of the  von Dyck group $D(a,b,c)$ are linked by a vertex-to-edge duality.
\end{abstract}

\maketitle
\tableofcontents

\section*{Introduction}
Let $D(a,b,c):=\Span{x,y\mid x^a=y^b=(xy)^c=1}$ be the von Dyck group, $\Gamma(a,b,c):=\Gamma(D(a,b,c),\{x,y\})$ its Cayley  graph corresponding to the generating set $\{x,y\}$, and $T(a,b,c):=T(D(a,b,c),\{H,K\})$ the rank two coset geometry determined by the subgroups $H:=\Span{x}$ and $K:=\Span{y}$. The latter will be regarded as a two-colored (i.e., bipartite) graph. We plan to  prove the following statements.
\begin{enumerate}
\item The natural action of $D(a,b,c)$   on the graph $T(a,b,c)$ is edge-regular and edge-transitive.\label{s1}
\item There is a  $D(a,b,c)$-equivariant bijection $b$ between $V_{\Gamma(a,b,c)}$, the set of vertices of $\Gamma(a,b,c)$, and $E_{T(a,b,c)}$, the set of edges of $T(a,b,c)$.\label{s2}
\item If $I(a,b,c)\subseteq E_{T(a,b,c)}^2$ denotes the set of incident pairs of edges of $T(a,b,c)$, then there is a map $\psi:I(a,b,c)\longrightarrow H\cup K$ such that the vertices  $d_1$ and $d_2$  of $\Gamma(a,b,c)$ are connected by an   $x$-colored (resp., $y$-colored) oriented edge   if and only if $\psi(b (v_1), b (v_2))=x$ (resp., $=y$). \label{s3}
\item There are uniform tilings $\T(a,b,c)$ and $\T'(a,b,c)$ of a constant curvature surface, such that the 1-skeleton of the former, supplied with a natural (vertex-)coloring, coincides with $T(a,b,c)$ and that of the latter, supplied with a natural edge-coloring and edge-orientation, identifies with $\Gamma(a,b,c)$.\label{s4}
\end{enumerate}
The purpose of (1) is merely to pave the way   for (2). 
The meaning of (2) is that $T(a,b,c)$ and $\Gamma(a,b,c)$ are linked by a vertex-to-edge duality, while (3) implies that all of the information about $\Gamma(a,b,c)$ is already contained in  $T(a,b,c)$, and this is the main result of the paper. Finally, (4) tells precisely that the passage from $T(a,b,c)$ to $\Gamma(a,b,c)$ can be done by means of manipulations of tilings; in particular, both $T(a,b,c)$ and $\Gamma(a,b,c)$ are planar graphs.

\begin{figure}
\epsfig{file=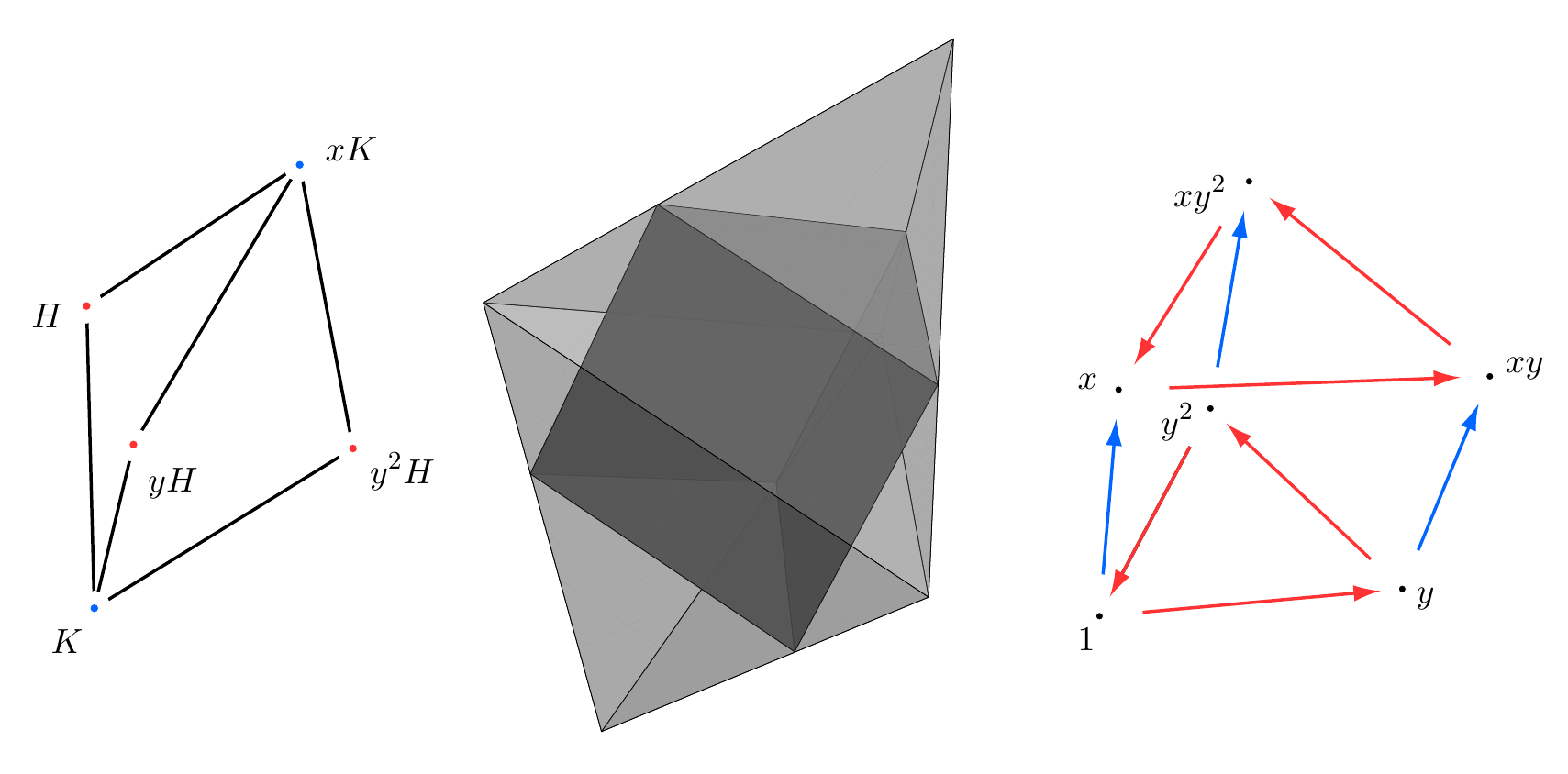,width=\textwidth}\caption{The coset geometry  $T$ of $\Z_6$ with respect to the subgroups $H:=\{0,3\}$ and $K:=\{0,2,4\}$ (left) and its Cayley graph $\Gamma$ with respect to the generators $x=3$ and $y=2$ (right). The    vertex-to-edge duality between $T$ and $\Gamma$ is artistically rendered in the center.  \label{figuragasante}}
\end{figure}

\subsection*{A toy model} Let $\Z_6=\Span{x,y\mid x^2=y^3=[x,y]=1}$ and $H$, $K$ be as above. Then the coset geometry  $T:=T(\Z_6,\{H,K\})$ is the graph appearing on the left of Fig. \ref{figuragasante}. It consists of three red vertices, $H, yH, y^2H$, and two blue ones, $K, xK$; an edge joins the cosets with nonempty intersection (i.e., \emph{incident}). The Cayley  graph $\Gamma:=\Gamma(\Z_6,\{x,y\})$ is displayed\footnote{Notice that,  not to overload the picture, the loops corresponding to the action of $x$ have been replaced by a unique (blue) arrow.}  on the right of   Fig. \ref{figuragasante}, and it consists of six (uncolored) vertices, corresponding to the six elements of $\Z_6$: now there is a blue oriented edge (i.e., a blue arrow) to indicate that the head of the arrow is obtained by acting by $x$ on its tail, and a red one when the action is that of $y$.

The center of Fig. \ref{figuragasante}  illustrates the vertex-to-edge duality this paper is concerned with: if the prism $\Gamma$ is fit into the double pyramid $T$, in such a way that the vertices of the former touch the middle points of the edges of the latter, then the natural $\Z_6$ action on $T$ determines that of $\Gamma$ and vice-versa. This example should clarify the meaning of the statements  \eqref{s2} and  \eqref{s3}.

\subsection*{Structure of the paper} In Section \ref{Sec1} we recall that $D(a,b,c)$ can be defined as the symmetry group of a tiling of a constant curvature surface; in Section \ref{Sec2} we review the construction of the coset geometry $T(a,b,c)$ and prove the statement \eqref{s1}; Section  \ref{Sec3} is dedicated to the Cayley  graph $\Gamma(a,b,c)$ and its duality with $T(a,b,c)$, i.e., the proof of the   statements \eqref{s2} and  \eqref{s3}: this duality is given a geometric interpretation in terms of derived tilings in  Section \ref{SecDuality}, thus proving the last statement \eqref{s4}. The concluding  Section \ref{Sec5} collects some applications and perspectives of the results obtained.

\subsection*{Conventions} If $G$ acts on $S$, $gs$ is the result of the action of $g\in G$ on $s\in S$.

\section{Von Dyck groups as symmetry groups of regular tilings}\label{Sec1}

The groups $D(a,b,c)$ were introduced by W. von Dyck in 1882 \cite{MR1510147}. They can be seen as the \virg{reflections-free} counterparts of the \emph{triangle groups}, thoroughly examined in Coxeter's book  \cite{MR0349820}, and   are tightly related to the theory of regular tilings of constant curvature surfaces. From now on, symbol $\S$ will denote either the sphere, the (real) plane $\R^2$, or the hyperbolic plane $\H$. By a \emph{tiling} $\T$ of $\S$ we mean a $CW$-complex structure on $\S$ whose 2-cells are polygons, such that each pair of them is either disjoint, or shares a whole   side, or   a single vertex: we identify $\T$ with the family of its 2-cells, henceforth called \emph{tiles}, so that $T\in \T$ means \virg{a tile $T$ of $\T$}. The 1-skeleton of  $\T$ is a planar graph embedded into $\S$, which we denote by $\partial\T$: accordingly, its set of vertices $V_{\partial\T}$ is the 0-skeleton of $\T$. We call $\T$   \emph{regular} if there is  a transitive action of an isometry subgroup of $\S$ on it.  When dealing with regular tilings, it is convenient to fix a \emph{basic tile}, say $T_0\in\T$: all the others can be obtained from $T_0$ by means of an isometry.\par

Following the standard terminology, $\T_{a,b,c}$ will denote the regular tiling of $\S$ whose basic triangle   $T_0$ has internal angles equal  to $\frac{\pi}{a}$,  $\frac{\pi}{b}$ and  $\frac{\pi}{c}$: depending on the sum of these values, the surface $\S$ must be elliptic, planar, or hyperbolic in order to accommodate $\T_{a,b,c}$ \cite{MR0349820,MR1393195,MR992195,MR769219}. The {triangle group} $\Delta(a,b,c)$ acts transitively on $\T_{a,b,c}$, whereas $D(a,b,c)$ acts on it by orientation-preserving isometries. We label the vertices of $T_0$ by $A$, $B$ and $O$,  $\frac{\pi}{a}$ (resp., $\frac{\pi}{b}$,  $\frac{\pi}{c}$) being the inner angle at $A$  (resp., $B$, $O$), in such a way that  the sequence of vertices $(A,B,O)$ occurs counterclockwise. Observe that $\partial\T_{a,b,c}$ is a 3-colored graph, with the coloring set given by the vertices of $T_0$: indeed,  there is a unique way to define a $\{A,B,O\}$-valued  coloring map on the whole $\partial\T_{a,b,c}$ which,   restricted to  $\{A,B,O\}$, reduces to the identity of $\{A,B,O\}$. If a vertex $v\in \partial \T_{a,b,c}$ has color $A$, we say that $v$ is of $A$-\emph{type} (and similarly for $B$ and $O$). A tile $T\in  \T_{a,b,c}$  is \emph{positively oriented} if running counterclockwise through its vertices, their colors appear in the same order as in $T_0$: hence,  $D(a,b,c)$  acts transitively on the subset $\T^+_{a,b,c}\subseteq \T_{a,b,c}$ of positively-oriented tiles (see the first row of Fig. \ref{figuragasante2}). This is a consequence of the fact that the generator $x$ (resp., $y$) corresponds to the rotation of $\T_{a,b,c}$ around the vertex $A$ of $T_0$ by an angle $\frac{2\pi}{a}$ (resp., around the vertex $B$ of $T_0$ by an angle $\frac{2\pi}{b}$) and, by composing such rotations, one can move $T_0$ to any other positively-oriented tile. Such a geometric interpretation of $D(a,b,c)$ can make some proofs more transparent.
 
\begin{lemma}\label{lemmaStupido1}
 The intersection $H\cap K$ is trivial.
\end{lemma}

\begin{proof}
 Suppose that $x^{r}=y^s\neq 1$: then the two tiles $x^r T_0$ and $y^s T_0$ must be the same. However, the former is obtained by rotating $T_0$ around $A$, and as such it must contain $A$ itself while the latter must contain $B$. Since an edge of $\partial\T_{a,b,c}$ determines a unique positively-oriented tile, $x^r T_0$ and $ y^s T_0$ must be the same and coincide with $T_0$, which is the unique element of $\T^+_{a,b,c}$ containing the edge $(A,B)$, which is a contradiction.
\end{proof}

\begin{figure}
\epsfig{file=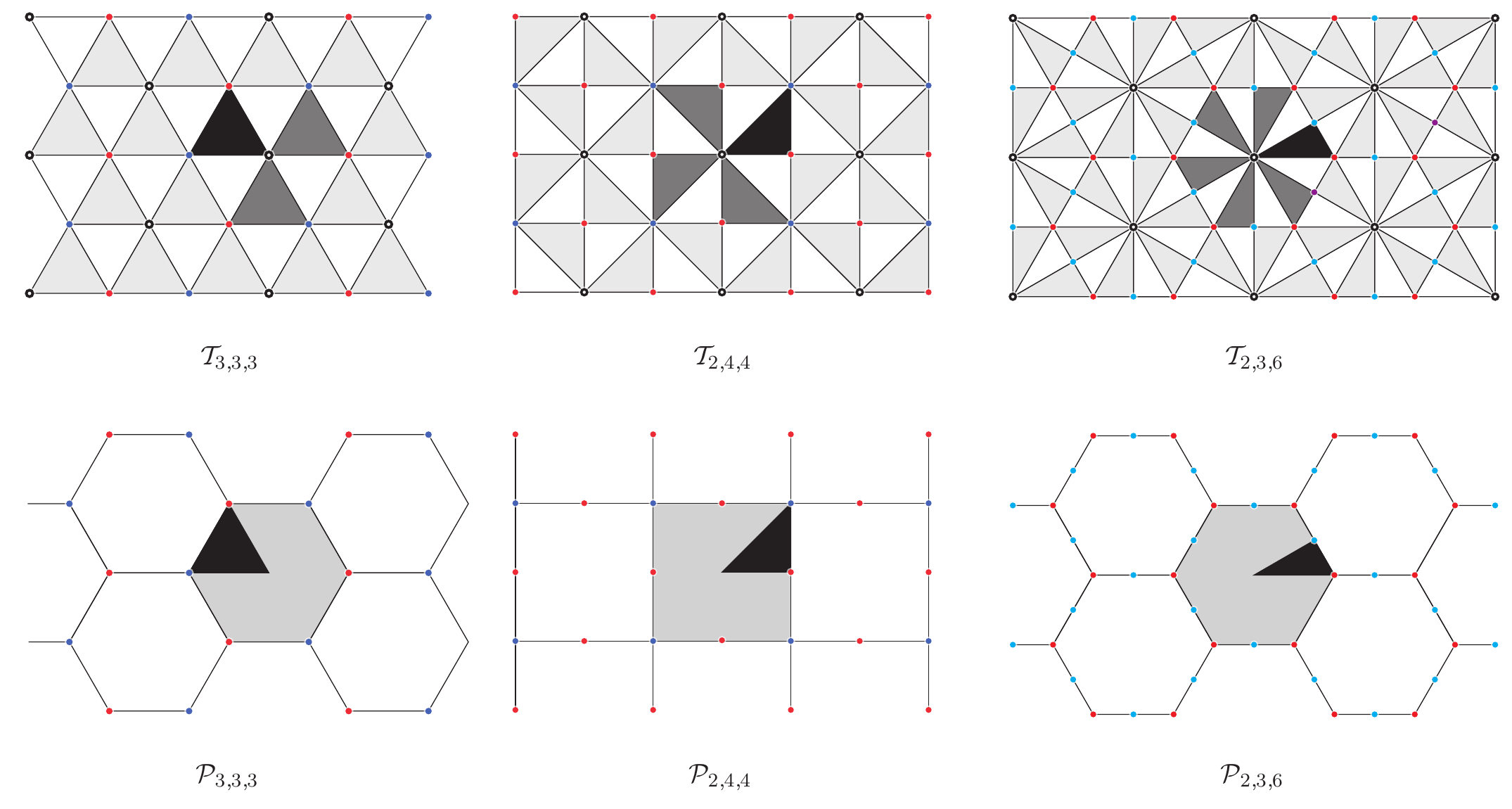,width=\textwidth}\caption{The first row displays the  only three cases of $\T_{a,b,c}$ for which $\S=\R^2$; darker triangles are the positively-oriented ones, i.e., the elements of $\T^+_{a,b,c}$.  The basic triangle is the full black one, whose vertices $A, B$ and $O$ are the red, the blue, and the empty ones, respectively. On the second row the corresponding polygonal tilings are pictured: the basic polygon $P_0$, which is the darker one, has 6, 8 and 12 sides, respectively, but it is  regular   only in the first case.  \label{figuragasante2}}
\end{figure}

Perform  now a full rotation of $T_0$  around its vertex $O$: the result is a $2c$-sided polygon $P_0$ whose boundary is a 2-colored (i.e., bipartite) subgraph of $\partial\T_{a,b,c}$ (see the first row of Fig. \ref{figuragasante2}). Then there exists a $2c$-polygonal tiling $\P_{a,b,c}$ of $\S$, such that any of its tiles is the union of the $2c$ tiles of $\T_{a,b,c}$ with a common  vertex of type $O$. Hence, its 1-skeleton $\partial \P_{a,b,c}$ is a bipartite graph and $V_{\partial \P_{a,b,c}}$ is the subset of $\partial \T_{a,b,c}$ consisting of the  vertices of type either $A$ or $B$. Since the tiles of $\P_{a,b,c}$ are in one-to-one correspondence with the $O$-type vertices of $\partial \T_{a,b,c}$, the von Dyck group $D(a,b,c)$ acts transitively\footnote{But not freely: the stabilizer of $P_0$ is the subgroup generated by $xy$.} on $\P_{a,b,c}$ as well, making it into a regular tiling.\par

Given a regular polygonal tiling $\P$ of $\S$, made of convex (not necessarily regular) $2n$-gons, there is a unique tiling $\P^\prime$ of $\S$, made of $2n$-gons and $n$-gons, such that its 1-skeleton $\partial\P^\prime$ consists of convex $2n$-gons whose vertices are the middle points of the edges of the tiles of $\P$ (see Fig. \ref{figuragasante3}). Observe that, for $n=3$, $\partial\P^\prime$ is precisely the \emph{derived graph} \cite{MR0262097} of $\partial\P$, while for $n>3$  it is   its subgraph, but with the same vertices, referred to as  \virg{medial} by some authors \cite{MR2181153}. The von Dyck group acts transitively both on  the set of $2n$-gonal tiles  of $\P^\prime$ and on the $n$-gonal ones. 

\begin{definition}
 $\P^\prime$ is the \emph{derived tiling} of $\P$.
\end{definition}

\begin{figure}
\epsfig{file=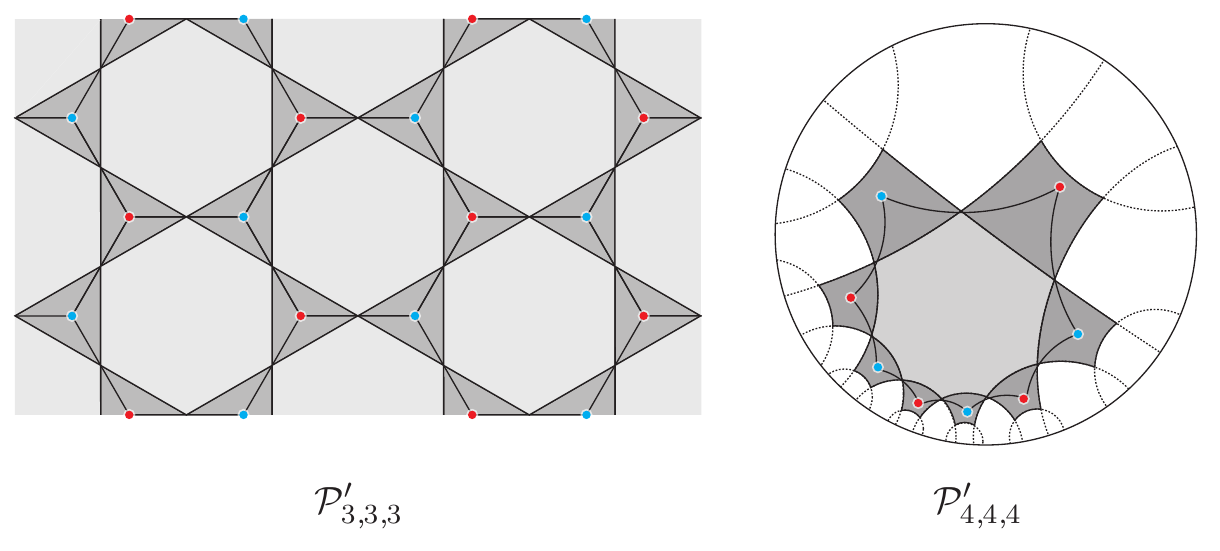,width=\textwidth}\caption{The derived tilings of the planar tiling $\P_{3,3,3}$  and the hyperbolic tiling $\P_{4,4,4} $  are shown, overlying the original tilings , and their (two-colored) 1-skeletons. The tiling $\P^\prime_{3,3,3}$ is known as a \emph{trihexagonal tiling} \cite{MR992195}.  \label{figuragasante3}}
\end{figure}

\section{Coset geometry of von Dyck groups}\label{Sec2}

The coset geometry
\begin{equation}\label{eqTits}
T(a,b,c):=\frac{D(a,b,c)}{H}\bigcup \frac{D(a,b,c)}{K}
\end{equation}
is the union of all $H$-cosets and $K$-cosets of $D(a,b,c)$. It is a classical example of a rank two coset geometry, proposed in 1962 by Tits    \cite{Tits62}, and it fits in the broader context of  incidence geometries, which were proposed to generalize   projective geometry but also found remarkable applications to problems of image recognition \cite{MR1360715,MR0140374,MR2951956,MR823824,MR823825,MR2984157}. But, for our purposes, it is more convenient to regard  $T(a,b,c)$ as a bipartite graph, whose vertices are the elements of $T(a,b,c)$, and two vertices are connected by an edge if the corresponding cosets have nontrivial intersection; this is precisely the construction of the \emph{intersection graph} associated to an incidence relation \cite{MR1672910}. Moreover, $T(a,b,c)$ is equipped with a natural coloring: the color (or\footnote{The term \virg{type} is usually adopted in the context of Tits geometries.} type) of the $H$- (resp., $K$-) cosets  is  \virg{$H$} (resp.,   \virg{$K$}).\par
Since $D(a,b,c)$ acts naturally on the two quotient sets appearing at the right hand side of  \eqref{eqTits}, the coset geometry $T(a,b,c)$ is equipped with a natural  $D(a,b,c)$-action. Furthermore,  $D(a,b,c)$ acts by colored graph transformations, since it sends edges to edges and preserves the type of a coset. The edge $(H,K)$ will be referred to as the \emph{basic edge}; we shall write an edge as a pair, where the first entry is always of type $H$.\par
Consider an edge $(d_HH, d_K  K)\in E_{T(a,b,c)}$,   act on it by an element $d\in D(a,b,c)$, and suppose that the resulting edge $d(d_HH,d_KK)=(dd_HH,dd_KK)$ has the same $H$-vertex as the original one, i.e., $(d_HH, d_K  K)$ has been \virg{rotated} around its vertex $d_HH$. This means that $dd_HH=d_HH$, i.e., $d$ stabilizes $d_HH$ and, as such, it belongs to $H^{d_H}$, which is   identified with $H$ via conjugation.  This proves  the next result.

\begin{figure}
\epsfig{file=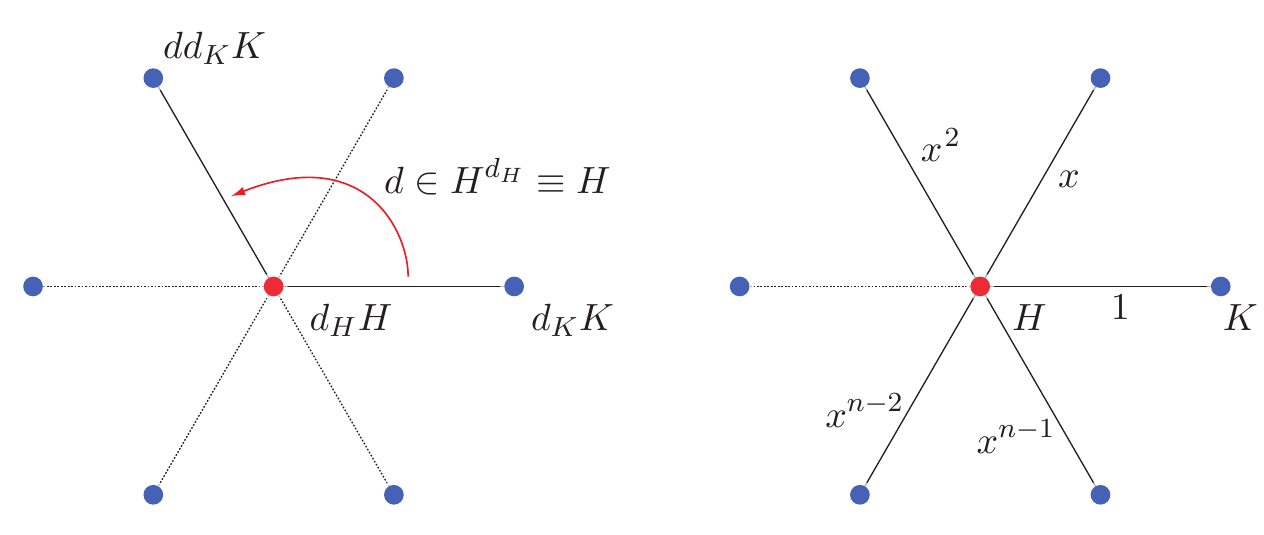,width=90mm}\caption{The subgroup $H^{d_H}$ acts transitively on the tree of the edges hinged at  $d_HH$, and the  action clearly behaves like  a rotation. This is an indication of the planar character of  the coset geometry.  \label{figuragasante5}}
\end{figure}

\begin{lemma}\label{lemStupido0}
All the edges of $T(a,b,c)$ having a common  $H$-type (resp., $K$-type) vertex, say $d_HH$ (resp., $d_KK$), can be     obtained by acting on a fixed one by  $H^{d_H}$ (resp., $K^{d_K}$). 
\end{lemma}
Observe that $D(a,b,c)$ acts edge-transitively on $T(a,b,c)$. Indeed, any edge  $e:=(d_HH, d_K  K)$  can be written as $e=d_H e^\prime$, with $e^\prime=(H,d_H^{-1}d_KK)$ and $e^\prime$ can in turn be  obtained by acting on the basic edge by an element of $H$ (Lemma \ref{lemStupido0} and   Fig. \ref{figuragasante5}). From Lemma \ref{lemmaStupido1} it follows also that the action is edge-regular: indeed, in view of edge-transitivity, the stabilizer of $e$ is conjugate to the stabilizer of the basic edge, which is the trivial intersection $H\cap K$.
\begin{corol}[Statement \eqref{s1}]\label{Stat1}
 $D(a,b,c)$ acts edge-regular and edge-transitively on $T(a,b,c)$; in particular, there is a unique way to label the edges of  $T(a,b,c)$ by the elements of  $D(a,b,c)$ such that the basic edge is labeled by 1 (see Fig. \ref{figuragasante4}).
\end{corol}

\begin{figure}
\epsfig{file=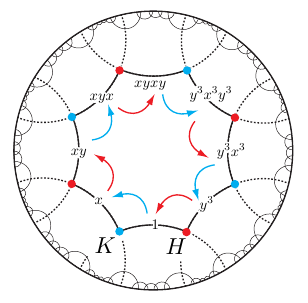,width=60mm}\caption{The 1-skeleton of the hyperbolic tiling $\P_{4,4,4} $   labeled by the elements of $D(4,4,4)$. The arrows represent the action of the generators of $D(4,4,4)$:  they turn out to constitute  the boundary of a tile of $\P^\prime_{4,4,4} $, corresponding to a cycle of $xy$ in  the Cayley graph.  \label{figuragasante4}}
\end{figure}
 
\section{The Cayley  graph of von Dyck groups}\label{Sec3}

According to a general result proved in 1958 by G. Sabidussi, the Cayley  graph $\Gamma(a,b,c)$    is, in a    sense, the unique edge-oriented and edge-colored graph on which $D(a,b,c)$  acts  in a vertex-regular and vertex-transitive way \cite{MR0097068}. In Corollary \ref{Stat1}  we proved that $D(a,b,c)$ acts  on the \hbox{(vertex-)}colored graph $T(a,b,c)$ in an edge-regular and edge-transitive way: then it is natural to suspect that $\Gamma(a,b,c)$ and  $T(a,b,c)$ can be obtained one from the other by, roughly speaking, \virg{replacing edges with vertices}.\par
Recall that, by definition, $V_{\Gamma(a,b,c)}=D(a,b,c)$ and that, for any $d\in D(a,b,c)$ there is a unique edge, call it $b(d)$, of $T(a,b,c)$ which is labeled by $d$ according to Corollary \ref{Stat1}. In other words, the map
\begin{eqnarray}
V_{\Gamma(a,b,c)} &\stackrel{b}{\longrightarrow} &E_{T(a,b,c)}\nonumber\\
d &\longmapsto & b(d)\label{eqMapB}
\end{eqnarray}
is bijective.
\begin{corol}[Statement \eqref{s2}]\label{Stat2}
The map $b$   defined by \eqref{eqMapB} is  $D(a,b,c)$-equivariant.
\end{corol}
\begin{proof}
 In order to prove that $b(d^\prime d)=d^\prime b(d)$, it suffices to observe  that if an edge $e$ of $T(a,b,c)$ is labeled by $d$, then the edge $d^\prime e$ is labeled by $d^\prime d$.
\end{proof}
If one tries  to reconstruct $\Gamma(a,b,c)$ out of $T(a,b,c)$, then Corollary \ref{lemStupido0} allows to obtain the vertices of the former out of the edges of the latter. It remains to describe the edges of $\Gamma(a,b,c)$. Also recall that   $\Gamma(a,b,c)$ is edge-colored and directed, so that not only its edges but their color and direction must be recovered as well.\par
Lemma \ref{lemStupido0}, together with the edge-regularity   of the $D(a,b,c)$-action, guarantees  the existence of a map
\begin{equation}\label{eqPSI}
\psi:I(a,b,c)\longrightarrow H\cup K
\end{equation}
assigning to any pair of incident  edges $(e_1,e_2)$ of $T(a,b,c)$ which have a common $H$-type (resp., $K$-type) vertex, the unique element $x^r\in H$ (resp., $y^s\in K$) such that the label of $e_2$  is the label of $e_1$ multiplied by  $x^r$ (resp., $y^s$). The existence of such an element   is clear from the tree displayed at the right of Fig. \ref{figuragasante5}:  $x^r$ is just the \virg{ratio}\footnote{More precisely, $r=n_2-n_1$, where    $x^{n_1}$ (resp., $x^{n_2}$) corresponds to $e_1$ (resp., $e_2$).} of $e_2$ by $e_1$.
\begin{corol}[Statement \eqref{s3}]\label{Stat3}\label{corSt3}
 There  is an $x$-colored (resp., $y$-colored) directed edge from the vertex $d_1$ to the vertex  $d_2$ of $\Gamma(a,b,c)$ if and only if $\psi(b(d_1),b(d_2))$ equals $x$ (resp., $y$).
\end{corol}
\begin{proof}
 Let us consider the $x$-colored case only. If $d_1$ is connected to $d_2$ by a directed edge in the Cayley  graph $\Gamma(a,b,c)$, it means that $d_2=d_1x$. By the definition \eqref{eqMapB} of the map $b$, $e_1:=b(d_1)$ is the edge $(d_1H, d_1K)$, while $e_2:=b(d_2)$ is the edge $(d_1H, d_2K)$. Hence, $(e_1,e_2)\in I(a,b,c) $ and, as such,  the map $\psi$ defined by \eqref{eqPSI} can be applied to it:   the result is $\psi (e_1,e_2)=x$, to be interpreted as the \virg{ratio} $\frac{d_2}{d_1}$. The converse is also true, due to the bijectivity of $b$.
\end{proof}
Fig. \ref{figuragasante} helps to visualize the above proof. Let $d_1=1$ and $d_2=x$ be vertices of the Cayley graph (displayed at the right): there is an $x$-colored directed edge between them (blue arrow). By the vertex-to-edge duality, portrayed in the center, this   edge of the Cayley graph is responsible for the fact that   the edge $e_2:=(H,xK)$ of the coset geometry (left) is obtained from $e_1:=(H,K)$ by acting on it by $x$. Formally, the datum $x$ can be recovered as $\psi(e_1,e_2)$, since $x$ is precisely the \virg{ratio} $\frac{e_2}{e_1}$.

\section{The duality between the Cayley  graph and the coset geometry in the context of tilings}\label{SecDuality}

In this section we prove the last statement \eqref{s4}, linking together the notions and  the results collected in the previous sections.\par
First, we establish a correspondence between the bipartite graphs $\partial\P_{a,b,c}$ and $T(a,b,c)$: define a map
\begin{eqnarray}
E_{T(a,b,c)}&\longrightarrow & E_{\partial\P_{a,b,c}}\nonumber\\
d(H,K)\equiv d&\longmapsto& d(A,B),\label{eqdAB}
\end{eqnarray}
where the identification $d(H,K)\equiv d$ is due to Corollary \ref{Stat1} and $(A,B)$ is one of the sides of the basic triangle $T_0$, defined in Section \ref{Sec1}. Since $\P_{a,b,c}$ is a regular tiling, the map \eqref{eqdAB} is surjective. It is also injective since an orientation-preserving isometry which fixes the segment $(A,B)$ must be identical. Hence, the abstractly defined graph $T(a,b,c)$ is the 1-skeleton of a concrete tiling of $\S$. On this surface, the duality between the Cayley  graphs and the coset geometry discussed in Section \ref{Sec3} can be recast in terms of the tiling  $\P_{a,b,c}$  and its derived tiling.\par
By the  definition of derived tiling (see Section \ref{Sec1}), there is a bijection
\begin{eqnarray}
E_{\partial \P(a,b,c)} &\stackrel{\mu}{\longrightarrow} &V_{\partial\P^\prime(a,b,c)}\nonumber\\
d  &\longmapsto & \mu(d)\nonumber\label{eqMapBeta}
\end{eqnarray}
where $\mu(d)$ is the middle point of the edge $d(A,B)$, and the edges of $\partial \P(a,b,c)$ are identified with the elements of $D(a,b,c)$ via \eqref{eqdAB}. Hence, $\mu\circ b$ establishes a one-to-one correspondence between the vertices of $\Gamma(a,b,c)$ and those of $\partial \P(a,b,c)$. Recall that the vertices $\mu(d_1)$ and $\mu(d_2)$ form an edge $e$ in  $\partial\P^\prime(a,b,c)$ if and only if $d_1(A,B)$ and $d_2(A,B)$ are incident and belong to the same tile of $\P(a,b,c)$ (see Fig \ref{figuragasante3}). Since $\partial \P(a,b,c)$ is a bipartite graph,  the edge $e$ can be given the same color of the vertex $v:=d_1(A,B)\cap d_2(A,B)$; moreover, the edge $e$ can be directed from $\mu(d_1)$ to $\mu(d_2)$ if the rotation sending $d_1(A,B)$ to $d_2(A,B)$ within the tile they belong to, is counterclockwise, and vice-versa.\footnote{In Fig. \ref{figuragasante4} such a tile is the hyperbolic basic octagon $P_0$, and $(A,B)$ is the edge labeled by \virg{$1$}; the vertex $v$ is a vertex of $\partial P_0$, and $d_1$ (resp., $d_2$) is the word coming before (in a clockwise sense) $v$ (resp., after it). A red (or blue) arrow, running counterclockwise, rotates the corresponding edge $\mu(d_1)$ on $\mu(d_2)$: hence, in the Caley graph, there is an oriented red (or blue) edge from $d_1$ to $d_2$. The eight vertices loop made by the red and blue arrows corresponds precisely to the element $xy$ of order four.} Suppose that this rotation is counterclockwise and that $v$ is of type $A$: then, in view of correspondence \eqref{eqdAB}, the edges $d_1(H,K)$ and $d_2(H,K)$ have the $H$-type vertex in common and $d_2=d_1x$, since, by definition, the inner angles of a   tiles of $\P(a,b,c)$ at its $A$-type vertices are $\frac{2\pi}{a}$.   Hence, in view of Corollary \ref{corSt3}, there is a directed $A$-colored edge from $\mu(d_1)$ to $\mu(d_2)$ in $\partial\P^\prime(a,b,c)$ if and only if there is a directed $x$-colored edge between $d_1$ and $d_2$ in $\Gamma(a,b,c)$.

\section{Applications and perspectives}\label{Sec5}
In the last Section \ref{SecDuality} we showed the relationship between two abstract procedures to associate a graph to the von Dyck group, in such a way that the group acts on it: undoubtedly, among them the Cayley  graph is a much more broadly exploited construction, being linked to the important notion of the \emph{genus of a group} (i.e., the smallest genus of a surface where its Cayley  graph can be embedded), introduced in 1972 by A.T. White \cite{MR0317980}. On the other hand, except for some sparse and marginal papers, there are no remarkable group theoretic applications of the theory of coset geometries. 
\subsubsection*{On the role of coset geometries} 
Now the coset geometry can be regarded as the 1-skeleton of a tiling on which the von Dyck group acts transitively, and as such it is linked to the notion of the \emph{strong symmetric genus} of a group (i.e., the smallest genus of a surface on which the group acts by orientation-preserving isometries \cite{MR701174}). It should be stressed that,  all constructions being $D(a,b,c)$-equivariant, the  results obtained descend to the factors of $D(a,b,c)$ which, as observed by P.M. Neumann in 1973, constitute in fact a very large family \cite{MR0333017}. In particular, we can recast (by using perhaps a simpler method)  a  result of Tucker \cite{MR701174}.
\begin{corol}[Tucker, 1983]\label{corTuck}
 Let $\frac{\S}{K}$ be a compact surface. Then the Cayley  graph of $\frac{D(a,b,c)}{K}$, with respect to the generating set $\{xK,yK\}$, is embedded into  $\frac{\S}{K}$ in a $\frac{D(a,b,c)}{K}$-invariant way.
\end{corol}
\subsubsection*{Genus and symmetric genus of factors of von Dyck groups}
In spite of the name \virg{duality} used before, the passage from a tiling to its derivative cannot be easily inverted. An easy consequence of Corollary \ref{corTuck} is that the genus of  $\frac{D(a,b,c)}{K}$ is bounded by its strong symmetric genus: a procedure to recover a tiling out of its derivative would allow to prove the converse as well.

\subsubsection*{Enumeration of the elements of von Dyck groups} 
Constructing the Cayley  graph of a group is the same as  solving the word problem for it \cite{MR2109550}. The duality between the Cayley  graph and the coset geometry discussed in Section \ref{SecDuality} may provide an effective way to construct the Cayley  graph of von Dyck groups and its factors. Consider, for example, the groups $D(n,n,n)$: they are  important since they cover the free Burnside groups $B(2,n)$, as noticed by some authors (see. e.g., \cite{MR823825,cryptoeprint2011398}). In this case, $\P_{2n}:=\P_{n,n,n}$ is made of regular $2n$-gons, and the whole tiling can be constructed   algorithmically by means of subsequent \virg{enlargements}  of the basic polygon $P_0$. This way, one can recursively define a map $e:\N_0\longrightarrow E_{\partial\P_{2n}}$ which is bijective, i.e., allows to enumerate the edges of $\partial\P_{2n}$ and, consequently, the elements of $D(n,n,n)$. We omit the details of such a construction, since they  can be easily reproduced.

\begin{figure}
\epsfig{file=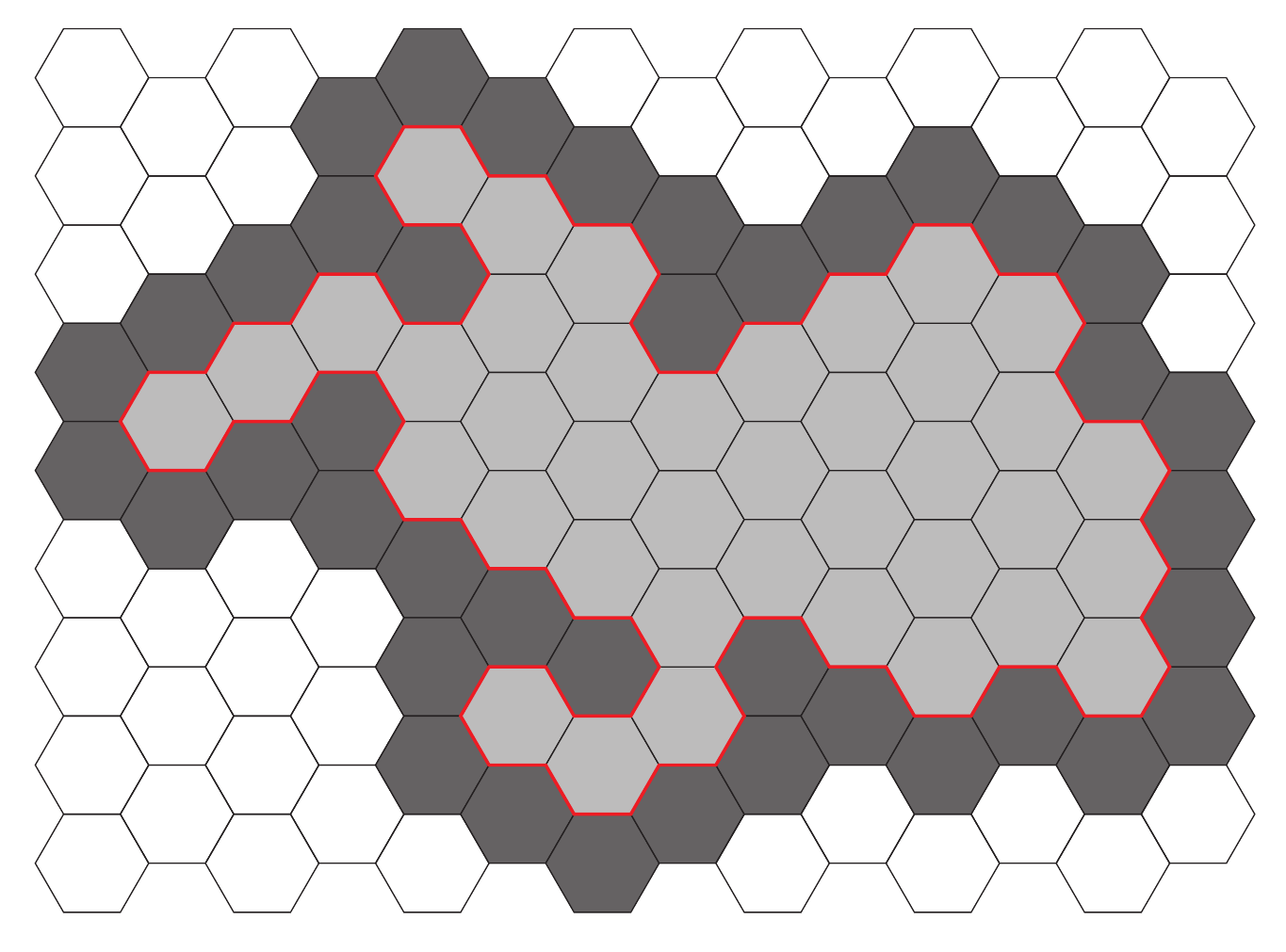,width=90mm}\caption{The enlargement of the sub-tiling $\Q\subseteq\P_{3,3,3}$ (soft grey) is accomplished by adding all the  hexagons (hard grey) which have  a vertex lying on the (not convex) polygon $\Gamma=\partial \Q$ (red line). The number of added hexagons depend on the internal angles of $\Gamma$, and it is given by formula \eqref{formulaAllargamento}.   \label{figuragasanteX}}
\end{figure}

\begin{corol}\label{corEnAlg}
There is a recursive way to enumerate the elements of $D(n,n,n)$, i.e.,   all its words can be  listed without repetitions.
\end{corol} 
\begin{proof}[Proof (a sketch)] Figure \ref{figuragasante4} shows the basis of the recursion: an injective map $e:\{0,1,\ldots, 2n-1\}\longrightarrow E_{\partial\P_{2n}}$  whose image is the boundary of the basic polygon, $\partial P_0$, i.e., a list  of the $2n$ words   appearing along it.\par  Now suppose that $e$ is defined on a finite subset of $\N_0$, such that its image is the boundary of a sub-tiling $\Q\subseteq\P_{2n}$ and $\Gamma$ is a simple\footnote{Non self-intersecting.} loop consisting of $N$ edges: in other words, it is   given a list (without repetitions) of all the edges contained into the $N$-gon $\Q$. Then it is easy to check that there are exactly 
\begin{equation}\label{formulaAllargamento}
N(n-1)-\sum_{k=1}^Ni_k
\end{equation}
 tiles of $\P_{2n}$ outside $\Q$ which intersect $\Q$, where $\frac{2\pi i_k}{n}$ is the internal angle of $\Gamma$ at its $k\Th$ vertex. Obviously, these tiles can be enumerated by the parameter $k=1,2,\ldots, N$ and, for each $k$, by another parameter in $1,2\ldots, n-i_k-1$ (the latter counts the external tiles hinged at the $k\Th$ vertex of $\Gamma$). Since the edges of a single tile can be listed, it is also possible to list all the edges contained into a sub-tiling $\widetilde{Q}$ which is obtained by adding to $\Q$ all incident tiles (see Fig. \ref{figuragasanteX}).
 \end{proof}
Corollary \ref{corEnAlg} becomes more interesting when it descends to the factors of $D(n,n,n)$, for instance, $B(2,n)$. It is well-known that the latter can be obtained by factoring   the former by the $n\Th$ powers subgroup $K_n:=D(n,n,n)^n$. In 1986 A.M. Vinogradov proposed an algorithmic way to check the finiteness of  Fuchsian $B(2,n)$, i.e., with $n>3$, based on the computation of a fundamental domain for $K_n$ in the hyperbolic plane \cite{MR823825} (see also \cite{Seward}).\footnote{The planar case is useful for understanding the behavior of the algorithm: in Figure \ref{fig2} a fundamental domain of $B(2,3)$ is represented as the intersection of the three bands in the plane bounded by the lines passing through the edges with the same color. The algorithm produces exactly such bands, and stops when the intersection becomes a closed polygon which, in this toy model, occurs after three steps.} In turn, to run such an algorithm, it is necessary to effectively list the elements of $K_n$: an enumeration of the elements of the  subgroup $K_n$ based on the result of Corollary \ref{corEnAlg} will be certainly more efficient than  the standard lexicographic method.

\begin{example}[A toy model: $B(2,3)$]    Up to isomorphisms, there are only 5 groups of order 27  \cite{MR1357169}, and a    unique one which is   meta-abelian without being abelian, has two generators, has exponent 3, and also possesses a cyclic derived subgroup:  
 \begin{equation}\label{presentazioneCarmela}
B(2,3)=\Span{x,y\mid x^3=y^3=[x,y,x]=[x,y,y]=1}.
\end{equation}
 The group \eqref{presentazioneCarmela} is the free Burnisde\footnote{A nice and exhaustive review on Burnisde problem can be found in Section 6.8 of H.S. Coxeter's book \cite{MR0349820}.} group   and (see \cite{cryptoeprint2011398,MR992077}) any of its element  $w$ can be written as 
\begin{equation}\label{eqB23elementi}
w=x^ay^b[x,y]^c,\quad a,b,c\in\{0,1,2\}.
\end{equation}
 Figure \ref{fig2} displays the coset geometry of $B(2,3)$ which, as a quotient of $\partial\P_{3,3,3}$, can be embedded in   a domain in $\R^2$ with some identifications on its boundary; each edge is labeled by the corresponding element of the group, according to the three parameters  \eqref{eqB23elementi}. The derived tiling $\T^\prime_{3,3,3}$ is also shown: the triangular cycles correspond to the generators $x$ and $y$, while the hexagonal ones correspond to their product $xy$.

\begin{figure}
 \epsfig{file=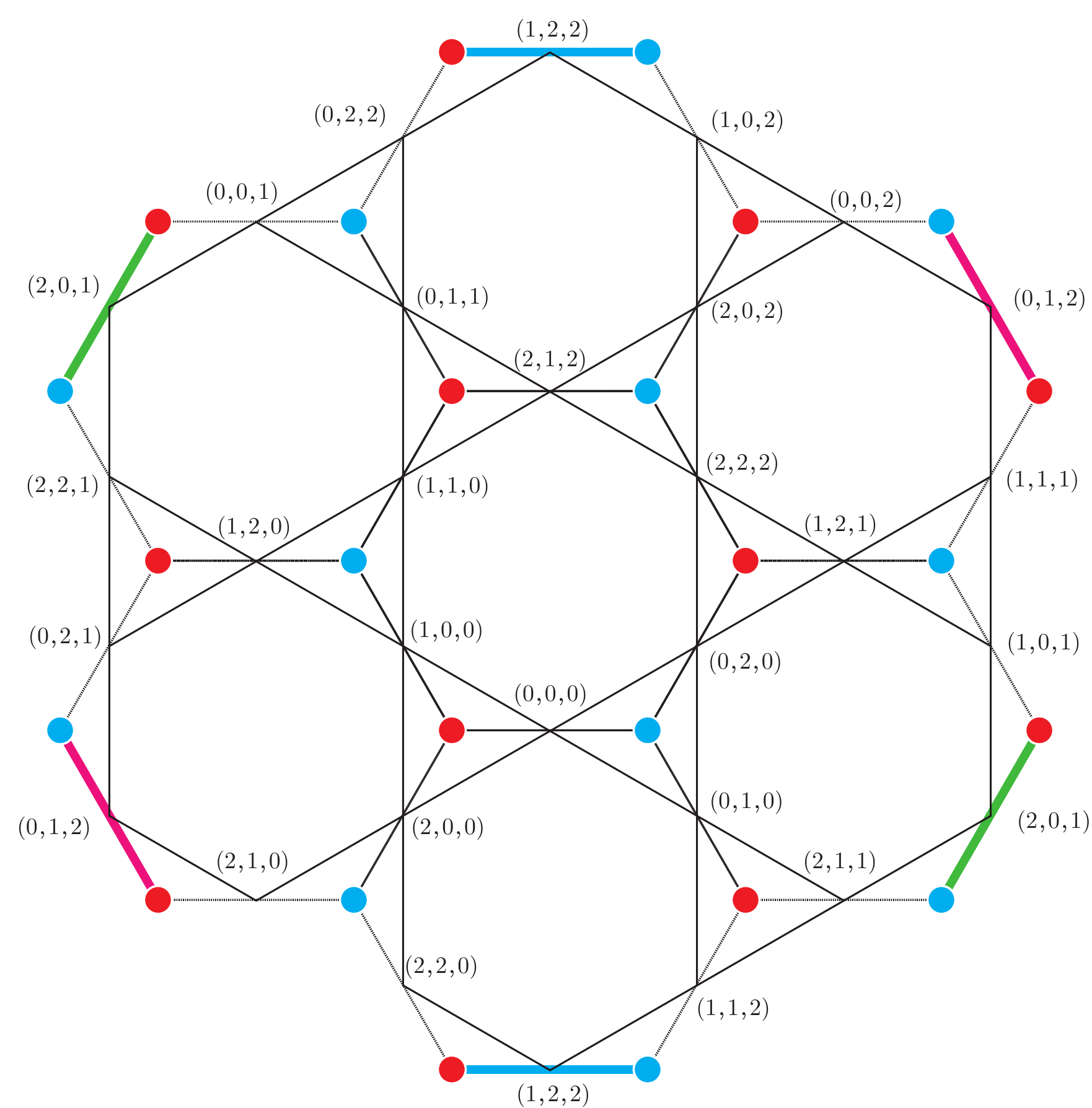,width=8cm}\caption{The coset geometry  of $B(2,3)$ overlaid by  its Cayley  graph, both embedded into   a plane domain with some boundary identifications, represented by  a pair of segments with the  same color. This is a geometric evidence of the finiteness of $B(2,3)$: there are exactly 27 edges in the coset geometry, which correspond to the 27 vertices of the Cayley graph. \label{fig2}}
\end{figure}

\end{example}

\subsubsection*{Beyond two generators}
Results of Section  \ref{Sec2} can be easily generalized to a group $G=\Span{S\mid R}$, with $|S|=m$, provided that the generating set $S$ is \emph{Borel-free}, i.e., $\underset{s\in S}{\cap}\langle s\rangle=1$, and $\Span{s}\neq 1$ for all $s\in S$. In this case, there holds the following generalization of Corollary \ref{Stat1}.
\begin{corol}
 The coset geometry $T(G,S)$ determined by the system of  $m$ subgroups $\{\Span{s}\}_{s\in S}$ is a rank $m$ incidence geometry or, equivalently, an $m$-colored graph, on which $G$ acts in a clique-regular and clique-transitive way.
\end{corol}
Notice that the notion of a \emph{clique} (which is a complete subgraph displaying all $m$ colors) has   replaced   that of an edge.\footnote{In the context of Tits geometry, the notion of a clique corresponds to that of a \emph{chamber} \cite{MR1332995}.} 
The vertex-to-edge duality discussed in Section \ref{SecDuality} becomes more complicated, even in a purely graph theoretic context, since it should be replaced by a \virg{vertex-to-clique} duality, which is a tougher concept,   yet worth further investigation. Framing it in a multi-dimensional tiling context is an even more challenging task (see, e.g, \cite{MR1232754} concerning a 3D example).

\subsubsection*{Final remarks}
We worked with the von Dyck group just because the absence of reflections makes everything easier; the passage to the full triangle group requires more care, but it can be done relying on standard techniques of double coverings.

 \bibliographystyle{amsplain}

 \bibliography{Bibliografia}

\end{document}